\documentclass[12pt]{article}

\usepackage[english]{babel}
\usepackage{graphicx}
\usepackage{latexsym}
\usepackage{amsfonts}
\usepackage{dsfont}
\usepackage{amscd}
\usepackage{amssymb}
\usepackage{amsmath}
\usepackage{mathtools}
\usepackage{mathrsfs}
\usepackage{amsthm}
\usepackage{tikz}

\input xy
\xyoption{all}

\renewcommand{\baselinestretch}{1.2}

\newcommand{\NN}{\mathbb{N}}

\newcommand{\re}{\mathbb{R}}

\newcommand{\T}{\mathcal{T}}
\newcommand{\p}{\mathcal{P}}
\newcommand{\x}{\mathbf{x}}
\newcommand{\bv}{\mathbf{v}}

\newcommand{\eps}{\varepsilon}

\newcommand{\uno}{\mathds{1}}
\newcommand{\gr}{\varphi}

\theoremstyle{definition}
\newtheorem{theo}{Theorem}[section]
\newtheorem{rmk}[theo]{Remark}

\newtheorem{ex}[theo]{Example}
\newtheorem{lem}[theo]{Lemma}
\newtheorem{defn}[theo]{Definition}
\newtheorem{cor}[theo]{Corollary}

\newtheorem{prop}[theo]{Proposition}

\begin{document}

\title{Division Point Measures from Primitive Substitutions}
\author{Daniel Gon\c{c}alves\thanks{Partially supported by CNPq, Brasil}, \and Charles Starling\thanks{Supported by CNPq, Brasil  }}

\date{}
\maketitle
\begin{abstract}
In this note we extend results of Olli concerning limits of point measures arising from substitutions. We consider a general primitive substitution on a finite polygon set in $\re^2$ and show that limits of certain atomic measures each converge to Lebesgue measure. 
\end{abstract}

\section{Introduction}

Associating point measures to iterated substitutions goes back to Kakutani \cite{Ka76} who considered division of the unit interval into two subintervals $[0,\alpha)$ and $[\alpha,1]$. If one repeats this division (with the same ratio for each interval created) and places a point mass at each division point, the limit of these measures converges weak-$*$ to a measure which is mutually singular with Lebesgue measure unless $\alpha =\, ^1\hspace{-0.05cm}/_{2}$, in which case it equals Lebesgue measure.

In \cite{Ol12}, Olli generalized this to a family of substitutions in $\re^2$, namely Conway's pinwheel substitution and Sadun's generalization of it (see \cite{sa98}). Conway's original pinwheel substitution is a scheme for dividing a $(1,2,\sqrt5)$-triangle into five uniformly scaled copies of the original, and Sadun's generalization provides a way to substitute any right triangle into five triangles which are similar to the original, but not uniformly scaled. Since the triangles created by these divisions are similar to the original, the substitution can be repeated, thereby dividing the original triangle into smaller and smaller triangles. At each step $k$ in this process, Olli defines three measures $\xi_k$, $\rho_k$ and $\omega_k$ on the original triangle $\mathcal{R}$ based on the substitution. These generate three sequencees of measures that are related to the distribution of vertices and triangles. The measure $\xi_k$ places a point mass at the barycenter of each triangle created at step $k$, while $\rho_k$ and $\omega_k$ place a point masses at the vertices of the triangles. The measure $\omega_k$ simply puts the same weight at each vertex, while $\rho_k$ weights each vertex for each triangle which has a vertex which intersects it. Each measure is then normalized so as to be a probability measure. Olli then proves (\cite{Ol12}, Theorem 1) that each of these measures converges to the same measure in the weak-$*$ topology, and that this limiting measure is equal to (normalized) Lebesgue measure if and only if the original triangle was a $(1,2,\sqrt5)$-triangle, that is, when the similar triangles at each step were uniformly scaled copies of $\mathcal{R}$.

We extend this result of Olli to other substitutions with uniform scaling. That is, we consider a finite set $\p$ of polygons in $\re^2$ and a substitution rule $\omega$ on $\p$ which divides elements of $\p$ into copies of elements of $\p$ scaled down by a constant factor of $\lambda>0$. We assume that $\omega$ is primitive in the sense that there is a $k$ such that for any $p,q\in\p$ if we divide $p$, $k$ times, then it will contain a copy of $q$. This situation is quite common in the theory of substitution tilings such as the Penrose tiling, see \cite{AP98} and \cite{So98} for example. We assume that the elements of $\p$ are disjoint, and let $X$ be the union of its elements, so that $X$ is a disjoint union of a finite number of polygons embedded in $\re^2$. We assume that the elements of $\p$ are uniformly scaled so that the Lebesgue measure of $X$ is 1. The substitution $\omega$ can then be repeatedly applied to the space $X$, dividing each polygon into smaller ones, and we can define probability measures analogous to those defined in \cite{Ol12}. Again, $\xi_k$ will place a point mass at a consistently chosen internal point of each polygon created at stage $k$, while $\rho_k$ and $\sigma_k$ will place point masses at the vertices of the polygons. The measure $\sigma_k$ puts a single point mass at each vertex, while $\rho_k$ again places weights on the vertices based on how many polygons intersect it (see Definitions \ref{defxi}, \ref{defrho}, and \ref{defsigma} for the precise statements, and see Example \ref{penrosemeasures} for a clear picture of the weights that these measures give). We note that we change the name of Olli's $\omega_k$ to $\sigma_k$ because the letter $\omega$ is usually reserved for the substitution rule in the literature concerning substitution tilings. We prove the following.
\begin{theo}\label{maintheorem}
Each of the sequences $\{\xi_k\}$, $\{\rho_k\}$ and $\{\sigma_k\}$ converges weak-$*$ to Lebesgue measure.
\end{theo}

This is consistent with \cite{Ol12}, as the only case there with a constant scaling factor is in the case of the original pinwheel tiling, and there the measures converge to Lebesgue measure.

This note is organized as follows. In Section 2 we present the background material needed to define our measures, as well as recall the Perron-Frobenius theorem. Since primitive substitutions in $\re^2$ arise most often in the study of aperiodic tilings, we use terminology common to that setting (for instance, we will call the polygons created at each stage {\em tiles}). In Section 3 we define our measures, and prove that $\{\xi_k\}$ and $\{\rho_k\}$ converge to Lebesgue measure using only the Perron-Frobenius theorem and general measure theory results. In Section 4 we use the fact that $X$ is a continuous surjective image of the path space of a simple Bratteli diagram arising from the substitution to show that $\{\sigma_k\}$ also converges to Lebesgue measure. 

\begin{rmk}
We are grateful to Ian Putnam for pointing out to us that the techniques used in \cite{KiPu12} can be adapted to prove our result. As is shown in \cite{AP98}, a primitive substitution gives rise to a mixing topological dynamical system $(\Omega, \phi)$, which is an example of a Smale space.  In this context, $\Omega$ has a measure of maximum entropy $\mu_{\Omega}$ for $\phi$, called the Bowen measure, which is a product of measures on the stable and unstable sets. In this Smale space the unstable sets are copies of $\re^2$, and in decomposing Bowen measure the factor on the unstable sets is Lesbesgue measure. In \cite{KiPu12} they prove that $\mu_\Omega$ is a weak-$*$ limit of a sequence of measures $\{\mu_{B,C}^k\}_{k\in\NN}$ constructed from a set $B$ in an unstable set of  $(\Omega, \phi)$ and a set $C$ in a stable set. Each $\mu_{B,C}^k$ is a sum of point masses depending on choice of $B$ and $C$.  If we choose $B$ to be any bounded open set in our $X$ and $C$ to be the set where point masses are present in our measures on $X$, then the sequences $\{\mu_{B,C}^k\}_{k\in\NN}$ eventually count up the same point masses as our sequences $\{\xi_k\}$, $\{\rho_k\}$ and $\{\sigma_k\}$ (different choices of $C$ result in the different sequences of measures). Our result that the sequences $\{\xi_k\}$, $\{\rho_k\}$ and $\{\sigma_k\}$ all converge to Lebesgue measure can be deduced from \cite{KiPu12}, Theorem 2.5.

We are also grateful to the referee for indicating that our results, as well as an analogous result for higher dimensions, also follow from results of Kellendonk and Savinien concerning spectral triples arising from stationary Bratteli diagrams, see \cite{KS12}, Sections 4.4 and 4.5. Spectral triples are an important and subtle concept in noncommutative geometry, see Connes book \cite{CoNCG} for a reference on spectral triples in particular and on noncommutative geometry in general.

Regardless of the above, our result is of independent interest, as it can be stated in basic measure-theoretic terms and can be proven with elementary techniques (notwithstanding the use of Bratteli diagrams in Section 4), avoiding the use of the heavy machinery of Smale spaces and noncommutative geometry.
\end{rmk}

\section{Substitutions in $\re^2$}

A {\em tile} is a polygon in $\re^2$, possibly carrying a label. A {\em patch} is a finite set of tiles whose interiors are pairwise disjoint. If $P$ is a patch, the {\em support} of $P$, denoted supp$(P)$, is the union of its tiles.  If $U\subset \re^2$, and $P$ is a patch, then we let $P(U)$ denote the set of tiles in $P$ which intersect $U$. We are concerned with substitutions on tiles. Let $\p = \{p_1, p_2, \dots, p_N\}$ be a finite set of polygons in $\re^2$, and without loss of generality assume that they are disjoint (so that, strictly speaking, $\p$ is itself a patch). In keeping with the literature on substitution tilings, we will call $\p$ the set of {\em prototiles}. For each $p\in\p$, pick a distinguished point $\x(p)$ in its interior and call this the {\em puncture} of $p$. If $t$ is a tile and $t = h(p)$ for some $p\in \p$ and isometry $h$ of the plane, we will say that $p$ is a {\em copy} of $p$, and that $\x(t) :=h(\x(p))$ is the puncture of $t$. We will assume that, perhaps after labeling, none of the prototiles are copies of each other. Denote by $\p^*$ the set of patches whose elements are copies of elements of $\p$. A {\em substitution} on $\p$ is a map $\omega: \p \to \p^*$ such that there exists $\lambda> 1$ such that supp$(\omega(p)) = \lambda p$ for all $p\in\p$. Clearly, $\omega$ can be extended to a map $\omega: \p^*\to\p^*$ in the obvious way, and so $\omega$ can be iterated. We say that $\omega$ is {\em primitive} if there exists $M\in\NN$ such that $\omega^M(p)$ contains a copy of $q$ for any $p,q\in\p$.

Fix once and for all a set $\p$ of prototiles and a primitive substitution $\omega$. Let $A$ be the $N\times N$ matrix whose $(i, j)$ entry is the number of copies of prototile $p_i$ in $\omega(p_j)$. Primitivity of $\omega$ implies primitivity of $A$, that is, there exists $M\in\NN$ such that the entries of $A^M$ are all strictly positive. The following are well-known facts from the Perron-Frobenius theorem on primitive matrices (see \cite{BS02}, Theorem 3.3.1 and Proposition 3.6.3 for example):

\begin{enumerate}\itemsep1pt
\item The matrix $A$ has a unique largest eigenvalue $\gamma>0$ which has multiplicity 1.
\item There is a unique left eigenvector $v_L$ of $A$ for $\gamma$ with positive entries such that $\sum v_L(i) = 1$.  We assume that all vectors are columns and so write $v_L^TA = \gamma v_L^T$.
\item There is a right eigenvector $v_R$ of $A$ for $\gamma$ with strictly positive entries and $v_L^Tv_R = 1$.
\item The matrix $v_Rv_L^T$ is the projection onto span$\{v_R\}$, and we have that
\[
\lim_{k\to\infty}\gamma^{-k}A^k = v_Rv_L^T.
\]
In particular, if $x\in\re^N$ then ${\displaystyle\lim_{k\to\infty}\gamma^{-k}A^kx = v_Rv_L^Tx.}$
\end{enumerate}
When the matrix $A$ arises from a substitution in $\re^2$, we also have the following:
\begin{enumerate}
\item[5] Assuming that the elements of $\p$ have been uniformly scaled so that the sum of their Lebesgue measures is 1, the $i$th entry of $v_L$, $v_L(i)$, is $m(p_i)$, the Lebesgue measure of $p_i$; in addition, $\gamma = \lambda^2$ (see \cite{So98}, Corollary 2.4).
\end{enumerate}
\begin{ex}\label{penrose}
Below is the Penrose substitution.\begin{center}
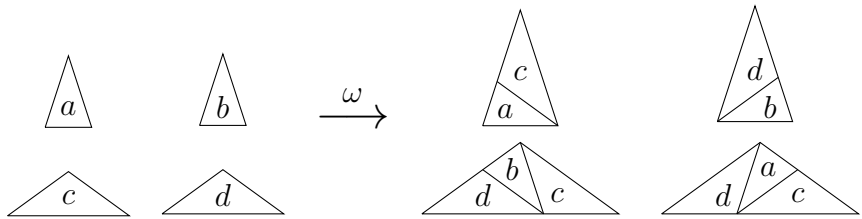
\begin{tabular}{ccccc}
\begin{tikzpicture}
\draw (0.309,0.25) node {$a$};
\draw (0,0) -- (0.618,0) -- (0.309,0.9510)-- (0,0);
\end{tikzpicture}
&
\begin{tikzpicture}
\draw (0,0) -- (0.618,0) -- (0.309,0.9510)-- (0,0);
\draw (0.309,0.25) node {$b$};
\end{tikzpicture}
& \Large $\stackrel{\omega}{\longrightarrow}$  &
\begin{tikzpicture}
\draw (0,0) -- (1,0) -- (0.5,1.5388)-- (0,0);
\draw (1,0) -- (0.191, 0.588);
\draw (0.3, .2) node {$a$};
\draw (0.5, 0.7) node {$c$};
\end{tikzpicture}
&
\begin{tikzpicture}
\draw (0,0) -- (1,0) -- (0.5,1.5388)-- (0,0);
\draw (0,0) -- (0.809, 0.588);
\draw (0.7, .2) node {$b$};
\draw (0.5, 0.7) node {$d$};
\end{tikzpicture}

\\
\begin{tikzpicture}
\draw (0,0) -- (1.618,0) -- (0.809,0.5877)-- (0,0);
\draw (0.809,0.25) node {$c$};
\end{tikzpicture}
&
\begin{tikzpicture}
\draw (0,0) -- (1.618,0) -- (0.809,0.5877)-- (0,0);
\draw (0.809,0.25) node {$d$};
\end{tikzpicture}
& &
\begin{tikzpicture}
\draw (0,0) -- (2.618,0) -- (1.309,0.9501)-- (0,0);
\draw (1.618,0) -- (0.809,0.5877);
\draw (1.618,0) -- (1.309,0.9501);
\draw (0.809,0.25) node {$d$};
\draw (1.8,0.25) node {$c$};
\draw (1.2,0.6) node {$b$};
\end{tikzpicture}
&
\vspace{0.5cm}
\begin{tikzpicture}
\draw (0,0) -- (2.618,0) -- (1.309,0.9501)-- (0,0);
\draw (1,0) -- (1.809,0.5877);
\draw (1,0) -- (1.309,0.9501);
\draw (1.809,0.25) node {$c$};
\draw (0.818,0.25) node {$d$};
\draw (1.418,0.6) node {$a$};
\end{tikzpicture}
\end{tabular}
\end{center}
Here, the scaling constant is the golden ratio $\gr =\, ^{(1 + \sqrt5)}\hspace{-0.05cm}/_{2} $, and the ratios of the three side lengths on the tiles on the left are 1, $\gr$, and $\gr^2$. We note that the copies on the right hand side are images of the original prototiles under orientation-preserving isometries; it is possible to present this substitution with just two prototiles but since the tiles have reflectional symmetries one would need to explicitly state the isometry for each tile on the right. In our case the substitution matrix is
\[
A =\left[\begin{array}{cccc}
1&0&0&1\\
0&1&1&0\\
1&0&1&1\\
0&1&1&1
\end{array}\right].
\] 
The entries of $A^2$ are all strictly positive, so $\omega$ is primitive. One calculates that the largest eigenvalue of $A$ is $\gr^2$, with left and right eigenvectors $v_L =\, ^1\hspace{-0.05cm}/_{2\gr^2}[\begin{array}{cccc}1&1&\gr&\gr\end{array}]$ and $v_R =\, ^{\gr^2}\hspace{-0.05cm}/_{(\gr^2+1)}[\begin{array}{cccc}1&1&\gr&\gr\end{array}]^T$ after normalizing. One sees that the entries of $v_L$ do indeed add up to $1$ and the ratios of its entries are exactly the relative areas of the prototiles.
\end{ex}

For more examples of substitutions in $\re^2$ we refer to \cite{AP98}, \cite{FreWeb} and \cite{GS13}. For more on the general theory of substitutions, we refer to Queff{\'e}lec, \cite{Q87}.
\section{Point Measures}

In \cite{Ol12}, Olli considers sequences of measures on a single right triangle which is then divided according to a generalized pinwheel scheme. She then shows that each of these sequences converges to the same measure. Because we have more than one tile type, we will instead define analogous measures on the union of the prototiles. Since we assumed that the prototiles are disjoint, the union of the prototiles embeds into $\re^2$ nicely.

Let
\[
X = \bigcup_{p\in\p}p
\]
be the disjoint union of the prototiles. All measures we consider are defined on the $\sigma$-algebra of the Borel sets in $X$ and so we make no further explicit reference to it. As in the previous section we assume that elements of $\p$ are uniformly scaled so that the Lebesgue measure $m$ of $X$ is equal to 1. Hence, $m(p_i) = v_L(i)$.

Let
\[
\T_k = \lambda^{-k}\omega^k(\p) = \{ \lambda^{-k}\omega^{k}(p)\mid p\in\p\}.
\]
Also let 
\[
\T = \bigcup_{k\geq 0}\T_k
\]
Then each $\T_k$ is a patch (which is the union of the patches $\lambda^{-k}\omega^k(p_i)$) consisting of copies of elements of $\lambda^{-k}\p$. Each of the patches $\T_k$ has support $X$. From the definition of the matrix $A$, we have that the $j$th entry of the vector $A^ke_i$ is the number of copies of $\lambda^{-k}p_j$ in the patch $\lambda^{-k}\omega^k(p_i)$. In particular, if we define \[\uno = [1 \ 1 \ \cdots \ 1]^T\] to be the column vector in $\re^N$ with 1 in every entry, and let $e_i$ be the $i$th standard basis vector in $\re^N$, then the total number of tiles in $\lambda^{-k}\omega^k(p_i)$ is $\uno^T A^ke_i$, and the total number of tiles in $\T_k$ is $\uno^T A^k\uno$.

An element $t\in\T_k$ is a scaled copy of an element of $\p$, and so we can speak of its puncture without confusion -- denote the puncture of $t$ by $\x(t)$ as above.

\begin{defn}\label{defxi}
For each $k\in\NN$ define a measure $\xi_k$ on $X$ by
\[
\xi_k = \frac{1}{|\T_k|}\sum_{t\in\T_k} \delta_{\x(t)}
\]
where $\delta_x$ indicates the point mass at the point $x\in X$.
\end{defn}
For a tile $t$, we let $\bv(t)$ denote the set of vertices of $t$. We also write $\bv(P)$ for the set of vertices in the patch $P$.

\begin{defn}\label{defrho}
For each $k\in\NN$ define a measure $\rho_k$ on $X$ by
\[
\rho_k = \frac{1}{{\displaystyle\sum_{t\in \T_k}}|\bv(t)|}\sum_{\substack{t\in \T_k\\ x\in\bv(t)}} \delta_x
\]
\end{defn}

For a Borel set $E$ and for each tile in $\T_k$ which intersects $E$, $\rho_k$ counts one for each vertex in the tile contained in $E$, and then normalizes. Hopefully without confusion, we let $\bv$ indicate the column vector in $\re^{N}$ whose $i$th entry is the number of vertices in the tile $p_i$. Then we see that, since the $j$th entry of the vector $A^k\uno$ is the number of copies of $\lambda^{-k}p_j$ in $\T_k$, we must have that $$\sum_{t\in \T_k}|\bv(t)| = \bv^T A^k\uno.$$

\begin{defn}\label{defsigma}
For each $k\in\NN$ define a measure $\sigma_k$ on $X$ by
\[
\sigma_k = \frac{1}{|\bv(\T_k)|}\sum_{ x\in\bv(\T_k)} \delta_x
\]
\end{defn}

The space $X$ can be seen as the disjoint union of all the prototiles. Each measure substitutes each prototile $k$ times (without scaling) and assigns point masses to points inside each prototile, and then normalizes. The measure $\xi_k$ (the distribution of prototiles after $k$ division steps) puts a point mass at the puncture of each tile. The measure $\sigma_k$ (the distribution of vertices after $k$ division steps) puts a point mass at each vertex in the patch, while $\rho_k$ (the distribution of vertices, counted with multiplicity, after $k$ division steps) puts a point mass at each vertex weighted by how many tiles intersect it.

\begin{ex}\label{penrosemeasures}
In the case of the Penrose tiling from Example \ref{penrose}, the space $X$ is the set of triangles on the left hand side. We will show what the measures $\xi_2$, $\rho_2$ and $\sigma_2$ look like on supp$(\{b\})\subset X$.

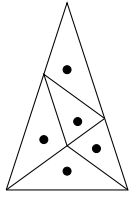
\begin{center}
\begin{tikzpicture}
\draw (0,0) -- (1.618,0) -- (0.809,2.490)-- (0,0);
\draw (0,0) -- (1.309, 0.951);
\draw (1.618,0) -- (0.809, 0.5877) -- (0.5 , 1.539) -- (1.309, 0.951);
\filldraw (0.809,0.25) circle (1.5pt);
\filldraw (0.809,1.6) circle (1.5pt);
\filldraw (1.2,0.55) circle (1.5pt);
\filldraw (.95,.91) circle (1.5pt);
\filldraw (.5,.67) circle (1.5pt);
\end{tikzpicture}
\end{center}

To calculate $\xi_2$ we first find $\uno^T A^2\uno = 26$. Then there is a point mass with weight $^1\hspace{-0.05cm}/_{26}$ at each puncture in the image above.

The measures $\rho_2$ and $\sigma_2$ put weights at the vertices of the tiles. Since each tile has three vertices, we have that $\bv^TA^2\uno = 3\cdot26 = 78$. Furthermore, after substituting twice the tiles $a$ and $b$ each have 6 vertices and tiles $c$ and $d$ have 9 vertices, and so $|\bv(\T_2)| = 30$. Thus $\rho_2$ and $\sigma_2$ put weights on the vertices according to the following image:

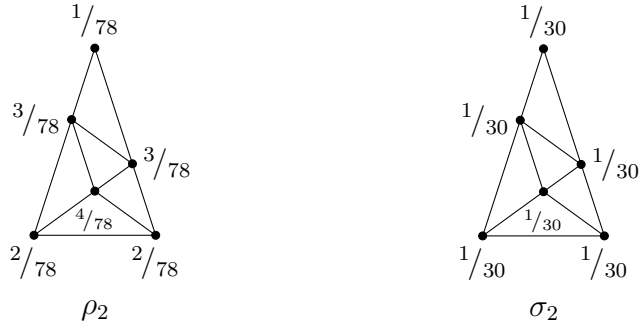
\begin{center}
\begin{tikzpicture}
\draw (0,0) -- (1.618,0) -- (0.809,2.490)-- (0,0);
\draw (0,0) -- (1.309, 0.951);
\draw (1.618,0) -- (0.809, 0.5877) -- (0.5 , 1.539) -- (1.309, 0.951);
\filldraw (0,0) circle (1.5pt);
\draw (0,0) node[anchor=north] {$^2\hspace{-0.05cm}/_{78}$};
\filldraw (1.618,0) circle (1.5pt);
\draw (1.618,0) node[anchor=north] {$^2\hspace{-0.05cm}/_{78}$};
\filldraw (0.809,2.490) circle (1.5pt);
\draw (0.809,2.490) node[anchor=south] {$^1\hspace{-0.05cm}/_{78}$};
\filldraw (1.309, 0.951) circle (1.5pt);
\draw (1.309, 0.951) node[anchor=west] {$^3\hspace{-0.05cm}/_{78}$};
\filldraw (0.5 , 1.539) circle (1.5pt);
\draw (0.5 , 1.539) node[anchor=east] {$^3\hspace{-0.05cm}/_{78}$};
\filldraw (0.809, 0.5877) circle (1.5pt);
\draw (0.809, 0.5) node[anchor=north] {{\scriptsize$ ^4\hspace{-0.05cm}/_{78}$}};
\draw (0.819, -1) node {$\rho_2$};
\end{tikzpicture}
\hspace{3cm}
\begin{tikzpicture}
\draw (0,0) -- (1.618,0) -- (0.809,2.490)-- (0,0);
\draw (0,0) -- (1.309, 0.951);
\draw (1.618,0) -- (0.809, 0.5877) -- (0.5 , 1.539) -- (1.309, 0.951);
\filldraw (0,0) circle (1.5pt);
\draw (0,0) node[anchor=north] {$^1\hspace{-0.05cm}/_{30}$};
\filldraw (1.618,0) circle (1.5pt);
\draw (1.618,0) node[anchor=north] {$^1\hspace{-0.05cm}/_{30}$};
\filldraw (0.809,2.490) circle (1.5pt);
\draw (0.809,2.490) node[anchor=south] {$^1\hspace{-0.05cm}/_{30}$};
\filldraw (1.309, 0.951) circle (1.5pt);
\draw (1.309, 0.951) node[anchor=west] {$^1\hspace{-0.05cm}/_{30}$};
\filldraw (0.5 , 1.539) circle (1.5pt);
\draw (0.5 , 1.539) node[anchor=east] {$^1\hspace{-0.05cm}/_{30}$};
\filldraw (0.809, 0.5877) circle (1.5pt);
\draw (0.809, 0.5) node[anchor=north] {{\scriptsize$ ^1\hspace{-0.05cm}/_{30}$}};
\draw (0.819, -1) node {$\sigma_2$};
\end{tikzpicture}
\end{center}
\end{ex}
For the rest of this section, we prove Theorem \ref{maintheorem} for $\{\xi_k\}$ and $\{\rho_k\}$. Our proof relies only on the Perron-Frobenius theorem and general measure theory results. We prove Theorem \ref{maintheorem} for $\{\sigma_k\}$ using the theory of Bratteli diagrams, see Section 4.

\begin{lem}(\cite{PU10}, Theorem 3.1.4)\label{nucontinuity}
If a sequence of Borel measures $\{\nu_k\}$ converges to $\nu$ weak-$*$ and if $E$ is a Borel set such that $\nu(\partial E) = 0$, then we have that $\nu_k(E) \to \nu(E)$. 
\end{lem}

\begin{lem}\label{edgemeasurezero}
Let $\{\nu_k\}$ be a sequence of measures on $X$ such that for any $t\in\T$ the sequence of numbers $\{\nu_k(t)\}$ converges to $m(t)$. Let $L$ be an edge of some tile in $\T_k$. Then $L$ has measure zero with respect to any weak-$*$ limit point of $\{\nu_k\}$. 
\end{lem}
\begin{proof}
Let $\nu$ be a cluster point of $\{\nu_k\}$. Let $\eps >0$ and find $\delta >0$ such that $m(B_\delta(L)) < \eps$. Find $l>0$ such that $\delta$ is greater than the diameter of any element of $\lambda^{-l}\p$. Then supp$(\T_l(L))\subset B_\delta(L)$.

Let $f$ be a continuous function with $0\leq f(x) \leq 1$ for all $x\in X$, such that $f(x) = 1$ for all $x\in L$ and also such that $f$ is supported on supp$(\T_l(L))$ (such a function exists by Urysohn's lemma). Then we have

\begin{eqnarray*}
\nu(L)& = & \int \chi_{L} \, \text{d}\nu  \leq \int f \, \text{d}\nu = \lim_{k\to\infty} \int f \, \text{d}\nu_k\\
      &\leq& \lim_{k\to \infty} \int \chi_{\text{supp}(\T_l(L))} \, \text{d}\nu_k  = \lim_{k\to \infty} \nu_k(\text{supp}(\T_l(L)))\\
      &\leq& \lim_{k\to \infty} \left(\sum_{t\in\T_l(L)}\nu_k(t) \right)= \sum_{t\in\T_l(L)}m(t) < \eps
\end{eqnarray*}
\end{proof}

\begin{rmk}
The above lemma is the analogous of Theorem 2 in \cite{Ol12}. Though the statement in \cite{Ol12} about $\{\xi_k\}$ is correct we do not see why the proof is.
\end{rmk}

\begin{prop}\label{convergence}
Let $\{\nu_k\}$ be a sequence of measures on $X$ such that for any $t\in\T$ the sequence of numbers $\{\nu_k(t)\}$ converges to $m(t)$. Then $\{\nu_k\}$ converges weak-$*$ to $m$.
\end{prop}
\begin{proof}
First we note that $\T$ is a countable set, and for any open set $U\subset X$ and $x\in U$ we can find $t\in\T$ such that $x\in t \subset U$. So $\T$ generates the Borel $\sigma$-algebra on $X$.

 Let $\nu '$ be a cluster point of $\{\nu_k\}$. By Lemma \ref{edgemeasurezero} $\nu '(\partial t) = 0$ for all $t\in \T$ and hence, by Lemma \ref{nucontinuity}, $\nu '(t) = m(t)$ for all $t\in \T$ and so $\nu ' = m$. 

Now, since the closed unit ball is compact in the weak-$*$ topology and $\{\nu_k\}$ has only one cluster point (Lebesgue measure) we conclude that $\{\nu_k\}$ converges to $m$ weak-$*$.
\end{proof}

As we see below, Proposition \ref{convergence} is a powerful tool to establish the convergence of $\{\xi_k\}$ and $\{\rho_k\}$.

\begin{lem}
If $t\in\T$ then the sequence $\{\xi_k(t)\}$ converges to $m(t)$.
\end{lem}
\begin{proof}
Let $t\in\T_j$. For $k>j$, $\xi_k(t)$ counts the number of tiles in the patch $\lambda^{-k+j}\omega^{k-j}(t)$ and divides it by the number of tiles among all $\T_k$. We show that this ratio converges to the Lebesgue measure of $t$ as $k\to\infty$.

We know that $t$ is a copy of some element of $\lambda^{-j}\p$, say $\lambda^{-j}p_s$. For $k>j$, the number of tiles in the patch $\lambda^{-k+j}\omega^{k-j}(t)\subset \T_k$ is equal to $\uno^T A^{k-j} e_s$. Hence we have
\begin{eqnarray}
\xi_k(t) &=& \frac{|\lambda^{-k+j}\omega^{k-j}(t)|}{|\T_k|} =\frac{\uno^T A^{k-j} e_s}{\uno^T A^k\uno} = \frac{\uno^T A^{k-j} e_s}{\uno^T A^j A^{k-j}\uno}\label{lebesgueeq}\\
         &\to&\frac{\uno^T v_Rv_L^Te_s}{\uno^T A^j v_Rv_L^T\uno}=\frac{\uno^T v_R \cdot v_L(s)}{\uno^T (\lambda^{2j} v_R)} = \frac{v_L(s)}{ \lambda^{2j}}\nonumber
\end{eqnarray}
Since $v_L(s)$ is the Lebesgue measure of $p_s$ for any $s$, this is equal to the Lebesgue measure of $p_s$ divided by $\lambda^{2j}$, which is the Lebesgue measure of $t$ (a copy of $\lambda^{-j}p_s$).
\end{proof}

\begin{cor}\label{Txi}
The sequence $\{\xi_k\}$ converges to $m$ weak-$*$.
\end{cor}

We now turn to $\rho_k$.

\begin{lem}\label{rhoedge}
Let $L$ be an edge in some tile in $\T$. Then the sequence of numbers $\{\rho_k(L)\}$ converges to 0.
\end{lem}
\begin{proof}
We first observe that if $L$ is the edge of some tile $t\in\T_j$, then the substitution divides $L$ into some number of subedges - for $k>j$, we denote this set of subedges as $\omega^{k-j}(L)$. If $\mathfrak{d}$ is the minimum edge length among all prototiles, then the number of elements of $\omega^{k-j}(L)$ is bounded by $\lambda^k\ell(L)/\mathfrak{d}$, where $\ell(L)$ denotes the length of $L$. 

Next, we know that $\rho_k$ assigns one point mass to $L$ for each tile vertex which intersects $L$, and then divides by the total number of vertices with multiplicity of tiles in $\T_k$. Because there are only a finite number of prototiles, it is possible to find $C>0$ such that the number of such point masses is bounded by $C\cdot |\omega^{k-j}(L)|$. Together with the previous paragraph, this implies that
\begin{eqnarray*}
\rho_k(L) &\leq&\frac{1}{\bv^T A^k\uno}\frac{C\lambda^k\ell(L)}{\mathfrak{d}} \\
          &\to&\frac{C\ell(L)}{\mathfrak{d}}\frac{\lambda^k}{\lambda^{2k}\bv^Tv_Rv_L^T\uno}\\
          &\to& 0
\end{eqnarray*}
\end{proof}
\begin{lem}
If $t\in\T$ then the sequence $\{\rho_k(t)\}$ converges to $m(t)$.
\end{lem}
\begin{proof}
Suppose that $t\in\T_j$, that $t$ is a copy of $\lambda^{-j} p_s\in\p$, and $k>j$. Then before normalizing, $\rho_k(t)$ counts $\bv(t')$ for each $t'\in\T_k$ contained in $t$, plus some contributions from tiles not contained in $t$ whose vertices intersect $t$. Since all these additional contributions must intersect the boundary of $t$, they must be dominated by $\rho_k(\partial(t))$, which by Lemma \ref{rhoedge} converges to zero. Hence we must have that  
\begin{eqnarray*}
\lim_{k\to\infty}\rho_k(t) &=& \lim_{k\to\infty} \frac{1}{\bv^TA^k\uno}{\displaystyle\sum_{t'\subset t, t'\in\T_k}\bv(t')} =\lim_{k\to\infty}\frac{\bv^T A^{k-j} e_s}{\bv^T A^k\uno}\\
         &=&\frac{\bv^T v_Rv_L^Te_s}{\bv^T A^j v_Rv_L^T\uno}=\frac{\bv^T v_R \cdot v_L(s)}{\bv^T (\lambda^{2j} v_R)} = \frac{v_L(s)}{ \lambda^{2j}}
\end{eqnarray*}
which, as before, is equal to $m(t)$.
\end{proof}
\begin{cor}\label{Trho}
The sequence $\{\rho_k\}$ converges to $m$ weak-$*$.
\end{cor}
\section{Convergence of $\{\sigma_k\}$}
To prove that $\sigma_k$ converges to $m$, we use the theory of Bratteli diagrams. For the reader's convenience we recall some of the key definitions below.

A {\em Bratteli diagram} is an infinite directed graph $B = (V, E, r, s)$, where $V$ denotes the vertex set, $E$ the edge set and $r, s$ are the range and source maps from $E$ to $V$. Furthermore, the vertex set $V$ is partitioned into disjoint finite sets $\{V_n\}_{n\in\NN}$, called {\em levels}, such that for each edge $e\in E$, $s(e)\in V_n$ implies that $r(e) \in V_{n+1}$ and also such that $r^{-1}(v)$ and $s^{-1}(v)$ are finite and nonempty for each $v\in V\setminus V_0$. The {\em path space} of $B$ is the set
\[
X_B = \{(x_i)_{i\in\NN}\mid x_i\in E, r(x_i) = s(x_{i+1}) \text{ for all } i\in\NN\}.
\]
If $E_n$ is the set of edges with $s(e)\in V_n$, then we can view $X_B$ as a subspace of the space $\prod E_n$ with each of the $E_n$ given the discrete topology.  Given a sequence $(x_1, x_2, \dots, x_n)$ with $x_1\in E_1$ and $r(x_i) = s(x_{i+1})$ for all $i$, let
\[
C(x_1, x_2, \dots, x_n) = \{(y_i)_{i\in\NN} \mid y_i = x_i \text{ for } 1\leq i \leq n\}.
\]
Sets of the above form are called {\em cylinder sets}; they are both open and closed in $X_B$ and generate the topology on $X_B$. 

Let 
\[
R_B = \{(x, y) \in X_B\times X_B \mid \text{ there exists } i\in\NN \text{ such that } x_k = y_k \text{ for all } k>i\}.
\]
Then $R_B$ is an equivalence relation on $X_B$; if $(x,y)\in R_B$ then $x$ and $y$ are said to be {\em tail equivalent}.

A Borel measure on $X_B$ is said to be {\em $R_B$-invariant} if whenever $C(x_1, x_2, \dots, x_n)$ and $C(x_1', x_2', \dots, x_n')$ are cylinder sets with $r(x_n) = r(x_n')$ then 
\[
\mu(C(x_1, x_2, \dots, x_n)) = \mu(C(x_1', x_2', \dots, x_n'))
\]

Suppose that $B = (V, E, r, s)$ is a Bratteli diagram and there exists $N\in\NN$ such that $|V_n| = N$ for all $n$. Write each of the $V_n$ as
\[
V_n = \{v_1^{(n)}, v_2^{(n)}, \dots, v_N^{(n)}\},
\]
and write $E(v, w)$ for the set of edges between vertices $v$ and $w$. Suppose that the cardinality of the set $E\left(v_j^{(n)}, v_i^{(n+1)}\right)$ depends only on $i$ and $j$. If this is all satisfied, we say that $B$ is {\em stationary}; the edges drawn between levels look the same at every level. In this case, let $A_B$ be the $N\times N$ matrix whose $(i, j)$ entry is the cardinality of the set $E\left(v_j^{(n)}, v_i^{(n+1)}\right)$; $A_B$ is called the {\em incidence matrix} of the stationary diagram. 

It is known (see for example \cite{Ha81}) that if $A_B$ is a primitive matrix, then there exists only one $R_B$-invariant probability measure $\mu_{A_B}$ on $X_B$. If $\alpha$ is the Perron-Frobenius eigenvalue of $A_B$ and $v_L$ is the left Perron-Frobenius eigenvector of $A_B$ normalized so its entries sum to 1, then 
\begin{equation}\label{invariantmeasure}
\mu_{A_B}(C(x_1, \dots, x_n)) = \alpha^{-n}v_L(r(x_n))
\end{equation}
where it is understood that $v_L$ can be seen as a function on the vertices. 

Now, let us return to our chosen primitive substitution $\omega$ on $\p$ as before. Construct from this data a stationary Bratteli diagram $B$ where at each level the vertex set is a copy of the prototile set, and there is an edge from $p_j$ to $p_i$ for each copy of $p_i$ in $\omega(p_j)$. Each edge $e$ with $s(e) = p_j$ and $r(e) = p_i$ labels a way that a copy of $\lambda^{-1}p_i$ sits inside $p_j$. The incidence matrix of this diagram $A_B$ is then therefore the same primitive matrix $A$ from the previous sections.  

Each edge in $B$ from $p_j$ to $p_i$ corresponds to a specific copy of $p_i$ in $\omega(p_j)$. Because of this, as in \cite{Ol12} we can associate to each tile $t$ in $\T_k$ a finite path $(x_1, \dots, x_k)$ in $B$. The edge $x_i$ corresponds to the scaled copy of $r(x_i)$ in $s(x_i)$ which contains $t$. Similarly, each point $x$ in $X$ is labeled by an infinite sequence of edges $(x_i)_{i\in\NN}$, and if $x$ lies on the edge of some tile then it is associated to a finite number of sequences. 


Define $\psi: X_B\to X$ by saying $\psi(x_i)_{i\in\NN}$ is the point given by the labeling $(x_i)_{i\in\NN}$. This map is continuous and finite-to-one. Furthermore, if $t$ is a tile in $\T_k$, then $\psi^{-1}(t)$ is the cylinder set given the labeling mentioned above. Recall that the measure $\mu_A$ defines a measure $\psi\mu_A$ on $X$ by the formula $\psi\mu_A(E) = \mu_A(\psi^{-1}(E))$ for any Borel set $E\subset X$. 

\begin{lem}
The measure $\psi\mu_A$ is equal to $m$.
\end{lem}
\begin{proof}

First, if $L$ is an edge of some tile from step $j$, and $k>j$, we have that $\psi^{-1}(L)$ is contained in $\psi^{-1}(\text{supp}(\omega^k(\T)(L)))$, which is itself contained in the union of cylinders in $X_B$ which correspond to tiles in $\omega^k(\T)(L)$. As in the proof of Lemma \ref{rhoedge} we can find $C>0$ such that the number of such cylinders is bounded by $C \ell(L)\lambda^k/\mathfrak{d}$. If $M$ is the maximum entry of the vector $v_L$, then by equation (\ref{invariantmeasure}) we have
\[
\psi\mu_A(L) \leq M\lambda^{-2k}C \ell(L)\lambda^k/\mathfrak{d}.
\]
This is true for all $k>0$, so $\psi\mu_A(L) = 0$.

Now, suppose $t$ is a tile. Then $\psi^{-1}(t)$ is the disjoint union of the cylinder which corresponds to $t$ and a subset of $\psi^{-1}(\partial t)$. The measure of the latter is zero, and so $\psi\mu_A$ agrees with $m$ on elements of $\T$. So, as before, they must be equal.
\end{proof}

We can now prove convergence of $\{\sigma_k\}$.

\begin{prop}\label{Tsigma} The sequence $\{\sigma_k\}$ converges to $m$ weak-$*$.
\end{prop}
\begin{proof}
Let $\sigma$ be a cluster point for $\{\sigma_k\}$, say $\sigma = \lim \sigma_{n_i}$. 

As before our first step is to show that the $\sigma$ measure of an edge is zero. For this, notice that the measures $\sigma$, $\sigma_k$ and $m$ can be seen as elements of the dual of $C(X)$. Also, (as in \cite{Ol12}) for all $k$ and all $f\in C(X)$, $\sigma_k(f)$ is between two constant multiples of $\rho_k(f)$ and so, since $\{\rho_k\}$ converges to Lebesgue measure, we have that $\sigma(f)$ is between two constant multiples of $m(f)$ for all $f\in C(X)$. Now, since the continuous functions are dense in $L^1(m)$, this last statement implies that $\sigma << m$ and hence the $\sigma$ measure of an edge is zero.

By Lemma \ref{nucontinuity}, $\sigma_{n_i}(t)$ converges to $\sigma (t)$ for all $t \in \T$. Also, we know that if $t$ and $t'$ are two tiles of the same type created at the same level $s$, then for $k > s$ we have  $\sigma_k(t) = \sigma_k(t')$. Thus $\sigma(t) = \sigma(t')$. 

Now, suppose that $\eta$ is a pullback of $\sigma$, that is, $\psi\eta = \sigma$ ($\eta$ exists because the set on which $\psi$ fails to be one-to-one has $\sigma$-measure zero). Then the above paragraph implies that whenever $C$ and $C'$ are cylinder sets which end at the same vertex, then we must have that $\eta(C) = \eta(C')$, that is, $\eta$ is $R_B$-invariant (this is because cylinders ending at the same vertex correspond to tiles of the same type created at the same level). Since $\mu_A$ is the only $R_B$-invariant measure we obtain that $\eta= \mu_A$ and hence $\psi\eta= \psi\mu_A$ and  $\sigma= m$ as desired.
\end{proof}
\begin{proof}[Proof of Theorem \ref{maintheorem}]
This combines Corollary \ref{Txi}, Corollary \ref{Trho} and Proposition \ref{Tsigma}. 
\end{proof}

\renewcommand{\baselinestretch}{1}
\scriptsize

\bibliography{C:/Users/Charles/Dropbox/Research/bibtex}{}
\bibliographystyle{plain}

\vspace{1.5pc}

Daniel Gon\c{c}alves and Charles Starling - Departamento de Matem\'atica, Universidade Federal de Santa Catarina, 88040-900 Florian\'opolis SC, Brazil. 

E-mails: Daniel Gon\c{c}alves \texttt{daemig@gmail.com}, Charles Starling \texttt{slearch@gmail.com}.

\vspace{0.5pc}

\end{document}